\documentclass[11pt,reqno]{amsart}
\usepackage{microtype}

\usepackage{mathtools}
\usepackage[dvipsnames]{xcolor}
\usepackage[colorlinks=true,linkcolor=Maroon,citecolor=OliveGreen]
{hyperref}
\usepackage{amssymb}
\usepackage[shortlabels]{enumitem}
\setlist[enumerate]{label={(\arabic*)}}

\usepackage{tikz-cd}

\usepackage[capitalize]{cleveref}

\usepackage[abbrev]{amsrefs}  

\numberwithin{equation}{section}
\newtheorem{theorem}{Theorem}[section]

\newtheorem{corollary}[theorem]{Corollary}

\theoremstyle{definition}
\newtheorem{definition}[theorem]{Definition}

\newtheorem{remark}[theorem]{Remark}

\def\Q{\mathbf{Q}} 
\def\Z{\mathbf{Z}} 
\def\F{\mathbf{F}} 

\def\K{\mathbf{K}}

\def\U{\mathbf{U}}

\newcommand\opr[1]{\operatorname{#1}}
\def\Aut{\opr{Aut}}

\def\GL{\opr{GL}}

\def\id{\opr{id}}

\def\Ind{\opr{Ind}}

\newcommand\Gal{\opr{Gal}}

\newcommand\gen[1]{\left\langle#1\right\rangle}

\newcommand\ceil[1]{\left\lceil#1\right\rceil}

\usepackage{todonotes}

\author[Eberhard]{Sean Eberhard}
\address{Sean Eberhard, Mathematics Institute, Zeeman Building, University of Warwick, UK}
\email{sean.eberhard@warwick.ac.uk}

\author[Guralnick]{Robert M. Guralnick}
\address{Robert M. Guralnick, Department of Mathematics, University of Southern California, Los Angeles, CA 90089-2532, USA}
\email{guralnic@usc.edu}

\author[Schildkraut]{Carl Schildkraut}
\address{Carl Schildkraut, Department of Mathematics, Stanford University, 450 Jane Stanford Way, Building 380, Stanford, CA 94305-2125, USA}
\email{carlsch@stanford.edu}

\author[Tracey]{Gareth Tracey}
\address{Gareth Tracey, Mathematics Institute, Zeeman Building, University of Warwick, UK}
\email{gareth.tracey@warwick.ac.uk}

\dedicatory{Dedicated to the memory of Avinoam Mann}

\thanks{%
S.E.\ is supported by the Royal Society through a University Research Fellowship (URF\textbackslash R1\textbackslash 221185).
R.G.\ was supported by a Simons Foundation Fellowship 00019819.
C.S.\ is supported by the NSF Graduate Research Fellowship Program under Grant No.~DGE-2146755.}

\begin{document}

    \title{Forcible groups and Frattini covers}

    \begin{abstract}
        We say that a finite group $G$ \emph{forces} a finite group $H$ if every finite cover of $G$ contains a subgroup isomorphic to $H$, and we say $H$ is \emph{forcible} if some finite group $G$ forces $H$.
        Complementing a classical result of Thompson--Mann, and answering a recent question of the third author, we show that a finite group is forcible if and only if it is abelian and its Sylow subgroups are elementary-by-cyclic.
        We also prove relative forcibility results for abelian $p$-groups in the settings of powerful $p$-groups and $p$-groups of bounded nilpotency class.
    \end{abstract}

    \maketitle

    \section{Introduction}

    In a finite group, to what extent can the presence of certain quotients force the presence of certain subgroups?
    This question is motivated by two well-known facts.
    First, if $G$ has a quotient isomorphic to a cyclic group $C_n$, then it has a subgroup isomorphic to $C_n$: this is obvious, since $G$ must contain an element of order divisible by $n$.
    Second, if $G$ has a quotient isomorphic to an elementary abelian group $C_p^{2n^2}$ then it has a subgroup isomorphic to $C_p^n$.
    This follows from classical results of Thompson and Mann: see \cite{huppert-thompson}*{III, 12.3} and \cite{mann} (see also \cite{HPPS}*{Theorem~1.1}).

    \begin{theorem}[Thompson--Mann]
        \label[theorem]{thm:thompson--mann}
        Let $P$ be a finite $p$-group with a maximal elementary abelian normal subgroup of rank $n$. Then
        \[
            d(P) \le n^2 + n(n+1)/2 \le 2n^2.
        \]
    \end{theorem}

    \begin{definition}
        Let $G$ and $H$ be finite groups. We say
        \begin{enumerate}
            \item $G$ \emph{forces} $H$ if every finite group $U$ with a quotient isomorphic to $G$ has a subgroup isomorphic to $H$, and
            \item $H$ is \emph{forcible} if some finite group forces $H$.
        \end{enumerate}
    \end{definition}

    Thus cyclic groups and elementary abelian groups are forcible. In this note we show that hardly any other groups are forcible.
    This answers a recent question of the third author~\cite{schildkraut}*{Question~12.11}.

    \begin{theorem}
        \label[theorem]{thm:main}
        A finite group $H$ is forcible if and only if $H$ is abelian and its Sylow subgroups are elementary-by-cyclic.
    \end{theorem}

    Here an abelian $p$-group is \emph{elementary-by-cyclic} if it is isomorphic to $C_{p^m} \times C_p^n$ for some $m, n \ge 0$.
    For example, the theorem asserts that $C_{p^{10}} \times C_p^{10}$ is forcible, while $C_{p^2} \times C_{p^2}$ is not.

    We also prove relative forcibility results for abelian $p$-groups in the settings of powerful $p$-groups and $p$-groups of bounded nilpotency class, viz.,
    if a powerful $p$-group covers $C_{p^{2f}}^d$ then it contains $C_{p^f}^d$,
    and if a $p$-group of class $c$ covers $C_{p^{2f+e+1}}^d$ where $e = \ceil{c \log_p(d)}$ then it contains $C_{p^f}^d$.

    We close the introduction with an additional question. We may call a collection $\mathcal H=\{H_i\}$ of finite groups \emph{forcible} if, for some finite group $G$, every finite group $U$ with a quotient isomorphic to $G$ has a subgroup isomorphic to some $H_i$. Does every finite forcible collection of groups contain a forcible group?

    \section{Frattini covers}

    Write $\Phi(G)$ for the Frattini subgroup of $G$.
    An epimorphism $\phi : U \to G$ of profinite groups is called a \emph{Frattini cover} if it satisfies one (hence all) of the following equivalent conditions~\cite{FJ}*{25.6.1}:
    \begin{enumerate}
        \item $\ker(\phi) \le \Phi(U)$,
        \item for every proper closed subgroup $H < U$, $\phi(H) \ne G$,
        \item for every subset $S \subseteq U$, $\overline{\gen S} = U$ if and only if $\overline{\gen{\phi(S)}} = G$.
    \end{enumerate}
    We are principally interested in finite groups, but it is customary and useful to extend the definition to profinite groups, as we will see.
    By (1), $\ker(\phi)$ is pronilpotent. The condition (2) shows that Frattini covers exist in abundance: given any epimorphism $\phi : U \to G$ of profinite groups, the minimal preimages of $G$ are Frattini covers of $G$.
    We may also observe that a composition of Frattini covers is again a Frattini cover.
    We call $\phi$ a \emph{Frattini $p$-cover} if $\ker(\phi)$ is pro-$p$.

    \begin{theorem}
        \label[theorem]{thm:op1}
        Let $G$ be a finite group and let $p$ be a prime. There is a Frattini $p$-cover $\phi : U \to G$ such that
        \begin{enumerate}
            \item the kernel $K = \ker(\phi) \cong C_p^d$ is elementary abelian,
            \item every element of $U$ of order $p$ is contained in $K$, and
            \item $U$ does not contain a copy of $C_{p^2} \times C_{p^2}$.
        \end{enumerate}
    \end{theorem}

    Except for property \emph{(3)}, this is essentially \cite{GT}*{Theorem~4.1}.
    The following proof is an embellishment on the same idea.

    An application of \emph{(1)--(2)} to the inverse Galois problem was noted in \cite{GT}*{Corollary 4.2}: taking $p = 2$, if $U = \Gal(L/\Q)$ where $L/\Q$ is a Galois extension then $G = \Gal(L^K/\Q)$ where $L^K/\Q$ is a totally real Galois extension.

    \begin{proof}
        For each subgroup $H \le G$ of order $p$, let $T_H$ be the one-dimensional trivial $\F_p H$-module. Recall that $\dim H^2(H, T_H) = 1$. Let $W_H = \Ind_H^G T_H$. Then Shapiro's lemma gives $H^2(H, T_H) \cong H^2(G, W_H)$; more precisely, the natural restriction-evaluation map
        \begin{equation}
            \label{eq:shapiro-map}
            H^2(G, W_H) \longrightarrow H^2(H, T_H)
        \end{equation}
        is an isomorphism.
        Let $0 \ne \alpha_H \in H^2(H, T_H)$ and let $\beta_H \in H^2(G, W_H)$ be the inverse image of $\alpha_H$.
        Let $W_1$ be the direct sum of all these modules $W_H$ and let $\beta_1 \in H^2(G, W_1)$ be the corresponding class (exactly as in \cite{GT}*{Theorem~4.1}).

        Similarly, for each subgroup $H \le G$ isomorphic to $C_p \times C_p$, again let $T_H$ denote the one-dimensional trivial $\F_p H$-module and let $W_H = \Ind_H^G T_H$.
        Choose a cohomology class $\alpha_H \in H^2(H, T_H)$ corresponding to a nonabelian central extension of $C_p \times C_p$ by $C_p$, and let $\beta_H \in H^2(G, W_H)$ be its inverse image under the Shapiro map \eqref{eq:shapiro-map}.
        Let $W_2$ be the direct sum of all these modules $W_H$ and let $\beta_2 \in H^2(G, W_2)$ be the corresponding class.

        Now let $W = W_1 \oplus W_2$ and $\beta = \beta_1 \oplus \beta_2$. By the usual correspondence between $H^2(G, W)$ and extensions of $G$ by $W$, we thus have an extension
        \[
            1 \longrightarrow W \longrightarrow E \longrightarrow G \longrightarrow 1
        \]
        corresponding to $\beta$.
        We claim that $W$ contains all elements of $E$ of order $p$, and that $E$ contains no subgroup isomorphic to $C_{p^2} \times C_{p^2}$.
        Indeed, suppose a subgroup $F \le E$ has image $H \cong C_p$ in $G$.
        Then we have a restricted extension
        \[
            1 \longrightarrow F \cap W \longrightarrow F \longrightarrow H \cong C_p \longrightarrow 1.
        \]
        By choice of $(W_1, \beta_1)$, this extension has a quotient of the form
        \[
            1 \longrightarrow T_H \longrightarrow \overline F \longrightarrow H \cong C_p \longrightarrow 1,
        \]
        where the corresponding cohomology class is $\alpha_H \ne 0$, which implies that $\overline F \cong C_{p^2}$, and thus $|F| \ge p^2$.
        Similarly, if $F \le E$ is a subgroup with image $H \cong C_p \times C_p$ in $G$, then we have a restricted extension
        \[
            1 \longrightarrow F \cap W \longrightarrow F \longrightarrow H \cong C_p \times C_p \longrightarrow 1,
        \]
        and, by the choice of $(W_2, \beta_2)$, this extension has a quotient of the form
        \[
            1 \longrightarrow T_H \longrightarrow \overline F \longrightarrow H \cong C_p \times C_p \longrightarrow 1
        \]
        whose corresponding cohomology class is $\alpha_H$.
        Hence $\overline F$ is nonabelian, and therefore so is $F$.
        In particular $F \ncong C_{p^2} \times C_{p^2}$: since all elements of order $p$ are contained in $W$, such a subgroup would have image in $G$ isomorphic to $C_p \times C_p$.

        Finally, let $U$ be a minimal preimage of $G$ in $E$. Then $U$ is a Frattini cover of $G$, $K = U \cap W$ is elementary abelian and contains all the elements of order $p$ in $U$, and $U$ does not contain a copy of $C_{p^2} \times C_{p^2}$.
    \end{proof}

    We remark that the direct sum operation for group cohomology used in the proof of \Cref{thm:op1} has a simple translation to the language of group extensions.
    Given extensions
    \[1\to W_i\to E_i\xrightarrow{\phi_i} G\to 1\]
    corresponding to $(W_i,\beta_i)$ for $i=1,2$, the extension corresponding to $(W_1\oplus W_2,\beta_1\oplus\beta_2)$ arises from the subdirect product
    \[E=\{(g_1,g_2)\in E_1\times E_2 : \phi_1(g_1)=\phi_2(g_2)\}.\]

    \begin{theorem}
        \label[theorem]{thm:op2}
        Let $G$ be a finite group, let $p$ be a prime, and let $q = p^f$ be a power of $p$, where $f \ge 1$. There is a Frattini $p$-cover $\phi : U \to G$ such that
        \begin{enumerate}
            \item $K = \ker(\phi) \cong C_q^r$ for some $r \ge 0$, and
            \item every element of $U$ of order dividing $q$ is contained in $K$.
        \end{enumerate}
    \end{theorem}
    \begin{proof}
        Let $\phi : \U\to G$ be the universal Frattini $p$-cover of $G$ and let $\K= \ker(\phi)$ (see \cite{FJ}*{25.12}).
        Then $\K$ and $\U$ are profinite groups, and $\K$ is free pro-$p$ on some $r\ge 0$ generators by \cite{FJ}*{25.12.8}.
        Now compare the cover $\phi_1 : U_1 \to G$ constructed in \Cref{thm:op1}, with $K_1 = \ker(\phi_1)$ elementary abelian.
        By the universal property of $\phi$, there is a cover $\theta : \U \to U_1$ making the following diagram commute:
        \[
        \begin{tikzcd}
            & \K \arrow[r, hook] \arrow[d, two heads]
            & \U \arrow[dr, two heads, "\phi"] \arrow[d, two heads, "\theta"] \\
            & K_1 \cong C_p^d \arrow[r, hook]
            & U_1 \arrow[r, two heads, "\phi_1"]
            & G
        \end{tikzcd}
        \]
        Since $K_1$ is elementary abelian, the kernel of $\theta|_\K$ contains $\Phi(\K) = [\K,\K] \K^p$. Now define
        \[
            K_2 = \frac{\K}{[\K,\K] \K^q}, \qquad U_2 = \frac{\U}{[\K, \K] \K^q}.
        \]
        Since $\K$ is free pro-$p$, we have $\K / [\K, \K] \cong \Z_p^r$ and $K_2 \cong C_q^r$.
        Thus we obtain the following commutative diagram of finite groups:
        \[
        \begin{tikzcd}
            & K_2 \cong C_q^r \arrow[r, hook] \arrow[d, two heads]
            & U_2 \arrow[dr, two heads, "\phi_2"] \arrow[d, two heads, "\eta"] \\
            & K_1 \cong C_p^d \arrow[r, hook]
            & U_1 \arrow[r, two heads, "\phi_1"]
            & G
        \end{tikzcd}
        \]
        The cover $\phi_2 : U_2 \to G$ is again Frattini since $[\K, \K]\K^q \le \ker \phi$.
        Now $\phi_1 : U_1 \to G$ has the property that if $g \in U_1$ and $\phi_1(g)$ has order $p$ then $g$ has order $p^2$.
        Since the kernel of $\eta|_{K_2}$ contains $K_2^p$, it follows that $\phi_2:U_2 \to G$ has the property that if $g \in U_2$ and $\phi_2(g)$ has order $p$ then $g^p \in K_2 \setminus K_2^p$, so $g^p$ has order $q$ and $g$ has order $pq$.
        Consequently every element of $U_2$ of order dividing $q$ is contained in $K_2$, and $\phi_2$ has the required properties.
    \end{proof}

    \begin{corollary}
        \label[corollary]{cor:op2}
        Let $G$ be a finite group and let $n \ge 1$ be an integer. There is a finite Frattini cover $\phi: U \to G$ such that
        \begin{enumerate}
            \item $K = \ker(\phi)$ is abelian, and
            \item every element of $U$ of order dividing $n$ is contained in $K$.
        \end{enumerate}
    \end{corollary}
    \begin{proof}
        Factorize $n = \prod_{i=1}^m q_i$, where $q_1, \dots, q_m$ are prime powers with distinct bases, and apply \Cref{thm:op2} to $q_1, \dots, q_m$ in sequence. We obtain a sequence of finite Frattini covers
        \[
        \begin{tikzcd}
            & U_m \arrow[r, two heads, "\phi_m"]
            & \cdots\arrow[r, two heads, "\phi_3"]
            & U_2 \arrow[r, two heads, "\phi_2"]
            & U_1 \arrow[r, two heads, "\phi_1"]
            & G,
        \end{tikzcd}
        \]
        with the property that every element of $U_i$ of order dividing $q_i$ is contained in the abelian group $K_i = \ker \phi_i$.
        The composition $\psi = \phi_1 \circ \cdots \circ \phi_m$ is again a Frattini cover.
        Being nilpotent, $K = \ker \psi$ is isomorphic to the direct product of its Sylow subgroups, so $K \cong \prod_{i=1}^m K_i$ is again abelian.

        Now we claim by induction that if $g \in U_i$ is an element of order dividing $n_i = q_1 \cdots q_i$ then $g \in \ker \psi_i$, where $\psi_i = \phi_1 \circ \cdots \circ \phi_i$.
        This is true for $i = 0$ with the natural conventions $n_0 = 1$, $U_0 = G$, and $\psi_0 = \id_G$.
        Now suppose $g \in U_i$ has order dividing $n_i$. Then $h = g^{n_{i-1}}$ has order dividing $q_i$, so $h \in K_i$ and $\phi_i(h) = 1$. Thus $\phi_i(g)$ has order dividing $n_{i-1}$, so $\phi_i(g) \in \ker \psi_{i-1}$ by induction, so $g \in \ker \psi_i$.
        This gives the result.
    \end{proof}

    The last ingredient we need to prove \Cref{thm:main} is the following variant of Thompson--Mann (\Cref{thm:thompson--mann}).

    \begin{theorem}
        \label[theorem]{thm:thompson--mann-variant}
        Let $G$ be a finite group with exponent divisible by $p^{m+e}$ and a quotient isomorphic to $C_p^{2n^2}$, where $p$ is a prime, $m, n \ge 1$, and $e = \ceil{\log_p(n)}$.
        Then $G$ contains a subgroup isomorphic to $C_{p^m} \times C_p^{n-1}$.
    \end{theorem}
    \begin{proof}
        By restricting to a Sylow subgroup we may assume $G$ is a $p$-group.
        Then, since $G$ covers $C_p^{2n^2}$, $d(G) \ge 2n^2$, so \Cref{thm:thompson--mann} implies that $G$ contains an elementary abelian normal subgroup $V$ of rank $r \ge n$,
        as well as an element $g$ of order $p^{m+e}$.
        Considering a chief series of $G$ containing $V$, it follows that $G$ has an elementary abelian normal subgroup $U \le V$ of rank $n$.
        Now $g$ induces an automorphism of $U$, but the maximal order of a $p$-element in $\Aut(U) \cong \GL_n(p)$ is $p^e$, so the element $h = g^{p^e}$ of order $p^m$ centralizes $U$.
        It follows that $\gen{U, h}$ is isomorphic to either $C_{p^m} \times C_p^{n-1}$ or $C_{p^m} \times C_p^{n}$ (depending on whether $\gen h \cap U = 1$).
    \end{proof}

    We are now ready to prove \Cref{thm:main}.

    \begin{proof}[Proof of \Cref{thm:main}]
        \emph{(only if)}
        Let $H$ be a forcible finite group. Let $G$ be a finite group that forces $H$.
        Applying \Cref{cor:op2} with $n = |H|$, there is a finite Frattini cover $\phi: U \to G$ with $K = \ker(\phi)$ abelian such that every element of $U$ of order dividing $n$ is contained in $K$. Since $G$ forces $H$, $H$ embeds in $U$ and the image must be contained in $K$, so $H$ is abelian.
        Similarly, \Cref{thm:op1} implies that $H$ does not contain a copy of $C_{p^2} \times C_{p^2}$ for any prime $p$.
        Thus each Sylow subgroup of $H$ must be elementary-by-cyclic.

        \emph{(if)}
        First, observe that, for any prime $p$, \Cref{thm:thompson--mann-variant} implies that an elementary-by-cyclic abelian $p$-group
        \[
        H = C_{p^m} \times C_p^{n-1} \qquad (m, n \ge 1)
        \]
        is forced by the abelian group
        \[
        G = C_{p^{m+e}} \times C_p^{2n^2}, \qquad e = \ceil{\log_p(n)}.
        \]
        Now suppose $H$ is an abelian group whose Sylow subgroups $H_p$ are elementary-by-cyclic. Let $P$ be the set of prime factors of $|H|$. For each $p \in P$ let $G_p$ be an abelian finite group that forces $H_p$, and let $G = \prod_{p \in P} G_p$.
        We claim that $G$ forces $H$.
        Indeed, let $\phi: U \to G$ be a cover of $G$.
        Without loss of generality $\phi$ is a minimal cover, hence Frattini, so $U / \Phi(U)$ is abelian.
        This implies that $U$ is nilpotent.
        Since $G_p$ forces $H_p$, it follows that $H_p$ embeds in $U$.
        Since $U$ is nilpotent, it is the direct product of its Sylow subgroups, so $H = \prod_{p \in P} H_p$ also embeds in $U$.
    \end{proof}

    \section{Relative forcibility results}
    \label{sec:relative}

    We complement our main result \Cref{thm:main} by proving forcibility of abelian groups within special classes of $p$-groups.
    For a $p$-group $G$, we denote as usual
    \[
        \Omega_e(G) = \gen{x\in G:x^{p^e}=1},
        \qquad
        \mho_e(G) =
        \gen{x^{p^e}:x\in G}.
    \]
    Recall that a finite $p$-group $G$ is said to be \emph{powerful} if
    $G/\mho_1(G)$ is abelian when $p$ is odd, and if $G/\mho_2(G)$ is abelian when
    $p=2$.
    \begin{theorem}
        \label[theorem]{thm:powerful}
        Let $U$ be a powerful finite $p$-group, and let $d, f \ge 1$. If $U \twoheadrightarrow C_{p^{2f}}^d$ then $U$ has a subgroup isomorphic to $C_{p^f}^d$.
    \end{theorem}

    \begin{proof}
        Let $X = \mho_f(U)$.
        Then $X$ is powerful~\cite{LM}*{1.2 \& 4.1.2} and $X \twoheadrightarrow C_{p^f}^d$.

        Suppose $x, y \in X$ have order at most $p^f$.
        By \cite{LM}*{1.7 \& 4.1.7}, there exist $a, b \in U$ such that \(x=a^{p^f}\), \(y=b^{p^f}\).
        Since $a$ and $b$ have order at most $p^{2f}$, \cite{FA}*{Theorem~1(ii)}, applied with $i=2f$ and $j=k=f$, gives
        \[
            [x, y] = [a^{p^f}, b^{p^f}] = 1.
        \]
        Thus $A = \Omega_f(X)$ is an abelian subgroup of $X$ of exponent at most $p^f$.
        Moreover $\Omega_i(A) = \Omega_i(X)$ for $0 \le i \le f$.
        Consequently $\exp(\Omega_i(X)) \le p^i$, so $\Omega_i(X)$ is precisely the set of elements of $X$ of order at most $p^i$.

        Write
        \[
            A\cong C_{p^f}^r\times B,
            \qquad \exp(B)\le p^{f-1}.
        \]
        Since $X \twoheadrightarrow C_{p^f}^d$, it follows that $\mho_{f-1}(X) / \mho_f(X) \twoheadrightarrow C_p^d$.
        Applying \cite{FA}*{Theorem~4}, $|\Omega_i(X)| = |X : \mho_i(X)|$ for $0 \le i \le f$, so
        \[
            p^r = |\Omega_f(A) : \Omega_{f-1}(A)|
            =
            |\mho_{f-1}(X) : \mho_f(X)| \ge p^d.
        \]
        Thus $r \ge d$. This implies that $A$, and hence $U$, contains a subgroup isomorphic to $C_{p^f}^d$.
    \end{proof}

    We now apply this to groups of bounded nilpotency class.

    \begin{theorem}
        \label[theorem]{thm:class}
        Let $c, d, f \ge 1$, let $U$ be a finite $p$-group of nilpotency class at most $c$, and
        let
        \(
            e=\ceil{c\log_p d}.
        \)
        If $U \twoheadrightarrow C_{p^{2f+e+1}}^d$
        then $U$ has a subgroup isomorphic to $C_{p^f}^d$.
    \end{theorem}

    \begin{proof}
        Replacing $U$ by a minimal preimage of $C_{p^{2f+e+1}}^d$, we may assume $d(U) = d$.
        Then $U$ has subgroup rank at most $d^c$.
        A Sylow $p$-subgroup of $\GL_{d^c}(p)$ has exponent $p^e$, so \cite{LM}*{1.13 \& 4.1.13} implies that $X = \mho_{e+1}(U)$ is powerful,
        and $X \twoheadrightarrow C_{p^{2f}}^d$, so \Cref{thm:powerful} implies that $X$, and hence $U$, has a subgroup
        isomorphic to $C_{p^f}^d$.
    \end{proof}

    \begin{remark}
        Together \Cref{thm:op1,thm:class} imply that the answer to \cite{HPPS}*{Question~3.8} is negative. For example, \Cref{thm:op1} implies that there is a $2$-group $U \twoheadrightarrow C_{2^7}^2$ not containing a copy of $C_4^2$, and \Cref{thm:class} implies that any such group has nilpotency class at least $3$.
        Similarly, Luca Sabatini has observed that $\texttt{SmallGroup}(3^6, 26)$ is a class-$3$ minimal cover of $C_9^2$.
    \end{remark}

    \section*{AI usage disclosure}

    The arguments in \Cref{sec:relative} were substantially simplified with the help of OpenAI's GPT-5.6 Sol.
    The same tool was used for routine literature search, language editing, and proofreading.
    The authors take full responsibility for the contents of the manuscript.

    \bibliography{refs}
\end{document}